\documentclass[a4wide,12pt]{article}
\usepackage[english]{babel}
\usepackage[sectionbib]{natbib}
\usepackage{amsmath}
\usepackage{amsfonts}
\usepackage{amsthm}
\usepackage{amssymb}
\usepackage{graphicx}
\usepackage{dsfont}
\def\sym{\fam\comfam\com}
\font\tensym=msbm10 \font\sevensym=msbm7 \font\fivesym=msbm5

\newfam\symfam
\textfont\symfam=\tensym \scriptfont\symfam=\sevensym
\scriptscriptfont\symfam=\fivesym
\def\sym{\fam\symfam\relax}

\def\R{{\sym R}}

\newtheorem{lemma}{Lemma}

\newtheorem{theorem}{Theorem}

\usepackage{epstopdf}
\usepackage[top=3cm , bottom=3.5cm , left=2.5cm ,right=2.5cm ]{geometry}
\usepackage{float}
\usepackage{authblk}

\def\Cc {\mathcal{C}}

\newcommand{\be}{\begin{eqnarray*}}
    \newcommand{\ee}{\end{eqnarray*}}
\newcommand{\ben}{\begin{eqnarray}}
\newcommand{\een}{\end{eqnarray}}
\begin{document}
\renewcommand{\thefootnote}{\arabic{footnote}}

    \begin{center}
        {\Large \textbf{Optimal Berry-Ess\'een    bound for Maximum likelihood estimation of the drift parameter in
         $ \alpha $-Brownian bridge}} \\[0pt]
        ~\\[0pt]
        Khalifa Es-Sebaiy\footnote{ Department of Mathematics, Faculty of Science, Kuwait University, Kuwait City,  Kuwait.
            E-mail: \texttt{khalifa.essebaiy@ku.edu.kw}\\ *Corresponding author}* and
            Jabrane Moustaaid\footnote{ National School of Applied Sciences, Cadi Ayyad University,
            Marrakesh, Morocco. E-mail: \texttt{jabrane.mst@gmail.com}}
        \\[0pt]~\\[0pt]
        \textit{  Kuwait University   and Cadi Ayyad University  }\\[0pt]
        ~\\[0pt]
    \end{center}
%\maketitle
    \begin{abstract}

    Let $T>0,\alpha>\frac12$. In the present paper we consider  the $\alpha$-Brownian
      bridge defined as  $dX_t=-\alpha\frac{X_t}{T-t}dt+dW_t,~ 0\leq t< T$,  where $W$ is a standard Brownian motion.
      We  investigate   the optimal rate of convergence  to normality of
     the maximum likelihood estimator (MLE)  for the parameter $ \alpha $ based on the continuous observation
$\{X_s,0\leq s\leq t\}$ as $t\uparrow T$. We prove that an optimal
rate of Kolmogorov distance for central limit theorem on the MLE  is
given by $\frac{1}{\sqrt{\vert\log(T-t)\vert}}$, as $t\uparrow T$.
First we compute an upper bound and then find a lower bound with the
same speed using Corollary 1  and Corollary 2 of \cite{kp-JVA},
respectively.
        \end{abstract}

   \noindent {\bf Keywords:}  $\alpha$-Brownian bridge, rate of convergence, MLE, Kolmogorov distance, Malliavin calculus.

   \noindent  {\bf 2020 Mathematics Subject Classification:}  60F05; 60H07; 62F12.

    \section{Introduction}\label{Sec}
    Fix a time interval $ \left[0,T \right)  $, with $T$ is a positive real number. We consider the
     $ \alpha $-Brownian bridge process $ X:=\lbrace X_t,~t\in\left[0,T \right)  \rbrace$, defined as the  solution to  the stochastic  differential equation
    \begin{eqnarray}
X_0=0;~~ dX_t=-\alpha\frac{X_t}{T-t}dt+dW_t,~~~ 0\leq t<
T,\label{bridge}
\end{eqnarray} where $W$ is a standard Brownian
motion, and   $ \alpha>0 $  is unknown parameter to be estimated.

In recent years, the study of various problems related to the
 $\alpha$-Brownian bridge (\ref{bridge})
has attracted interest. The process (\ref{bridge}) has been first
considered by \cite{BS}, where it is used to describe the evolution
of the simple arbitrage opportunity associated with a given futures
contract in the absence of transaction costs. For more information
and further references concerning the subject, we refer the reader
to    \cite{BP}, as well as \cite{Mansuy} and \cite{GT}.

An example of interesting problem related to $X$ is the statistical
estimation of $\alpha$ when one observes the whole trajectory of
$X$. A natural candidate  is the maximum likelihood estimator (MLE),
which can be easily computed for this model, due to the specific
form of (\ref{bridge}): one gets
\begin{equation}\label{MLE}
\widetilde{\alpha}_t=
-\left(\int_0^t\frac{X_u}{T-u}\,dX_u\right)\bigg/\left(\int_0^t\frac{X^2_u}{(T-u)^2}\,du\right)
,\quad  t<T.
\end{equation}
In (\ref{MLE}), the integral with respect to $X$ must of course be
understood in the It\^o sense. Moreover, it is easy to find
   \begin{eqnarray}\label{MLE2}
\alpha - \widetilde{\alpha}_t &= &
\left(\int_0^t\frac{X_u}{T-u}\,dW_u\right)\bigg/\left(\int_0^t\frac{X^2_u}{(T-u)^2}\,du\right).
\end{eqnarray}
 The asymptotic behavior of the MLE
$\widetilde{\alpha}_t$ of $\alpha$ based on the observation
$\{X_s,0\leq s\leq t\}$ as $t\uparrow T$ has been studied in
\cite{BP}.  Let us describe what is known about this problem: as
$t\uparrow  T$,
\begin{itemize}
\item[\textbullet] if $ \alpha>0 $, the MLE $\widetilde{\alpha}_t$   is strongly consistent, that is,
$\widetilde{\alpha}_t\longrightarrow\alpha$ almost surely, see
[\cite{BP}, Theorem 16];
\item[\textbullet] if $ 0<\alpha<\frac{1}{2} $\be
\left( T-t\right)^{\alpha-\frac{1}{2}}\left(
\alpha-\widetilde{\alpha}_t\right) {\overset{law}{\longrightarrow}}
~T^{\alpha-\frac{1}{2}}\left( 1-2\alpha\right)\times \Cc(1) \ee with
$\Cc(1)$ the standard Cauchy admitting a density function
$\pi^{-1}(1+x^2)^{-1},\ x \in \R$, see [\cite{BP},  Theorem 7];
\item[\textbullet] if $ \alpha=\frac{1}{2} $\be
\left|\log\left(T-t\right) \right| \left(
\alpha-\widetilde{\alpha}_t\right) {\overset{law}{\longrightarrow}}~
\frac{\displaystyle \int_0^TW_sdW_s}{\displaystyle \int_0^TW^2_sds},
\ee see [\cite{BP}, Theorem 5];
\item[\textbullet] if $ \alpha>\frac{1}{2} $
\be \sqrt{\left|\log\left(T-t\right)\right|}\left(
\alpha-\widetilde{\alpha}_t\right) {\overset{law}{\longrightarrow}}~
\mathcal{N}\left( 0,2\alpha-1\right), \ee see [\cite{BP}, Theorem
10].
\end{itemize}
The study of the asymptotic distribution of an estimator is not very
useful in general for practical purposes unless the rate of
convergence is known. To our knowledge, no result of the
Berry-Ess\'een type is known for the distribution of the MLE
$\widetilde{\alpha}_t$  of the drift parameter $\alpha$ of the
$\alpha$-Brownian bridge (\ref{bridge}). The aim of the present work
is to provide, when $ \alpha>\frac{1}{2} $,  an optimal rate of
Kolmogorov distance for central limit theorem of the MLE
$\widetilde{\alpha}_t$ in the following sense: There exist constants
$0 < c < C < \infty$, depending only on $\alpha$ and $T$, such that
for all $t$ sufficiently near $T$,
\begin{eqnarray*}
\frac{c}{\sqrt{\vert\log\left( T-t\right)\vert}}\leq\sup_{z\in
\mathbb{R}}\left\vert \mathbb{P}\left(\sqrt{\frac{\vert\log\left(
T-t\right)\vert}{2\alpha-1}}\left(
\alpha-\widetilde{\alpha}_t\right)\leq z\right)-\mathbb{P}\left(
Z\leq z\right)\right\vert\leq \frac{C}{\sqrt{\vert\log\left(
T-t\right)\vert}},
\end{eqnarray*}
where  $Z$ denotes a standard normal random variable.

Let us recall the case of the Ornstein-Uhlenbeck process defined as
solution to the equation $dX_t=-\theta X_t dt + dW_t,\ t\geq0,\
X_0=0$, with $\theta>0$. While \cite{bishwal} obtained the upper
bound $O(1/\sqrt{T})$ in Kolmogorov distance for normal
approximation of the MLE of the drift parameter $\theta$ on the
basis of continuous observation of the process ${X_t}$ on the time
interval $[0,T]$,     a lower bound with the same speed has been
recently obtained by \cite{KP-KJSS}.  This means that
$O(1/\sqrt{T})$ is an optimal Berry-Ess\'een  bound for the MLE of
$\theta$. Finally, we mention that \cite{EN} studied the parameter
estimation for so-called $\alpha$-fractional bridge which is given
by the equation (\ref{bridge}) replacing the standard Brownian
motion $W$ by a fractional Brownian motion.\\

We proceed as follows. In Section 2 we give the basic tools of
Malliavin calculus needed throughout the paper. Section 3 contains
our main result, concerning  the optimal rate of convergence to
normality of the MLE $\widetilde{\alpha}_t$.

\section{Preliminaries}
In this section, we recall some elements from stochastic analysis
that we will need in the paper. See \cite{NP-book}, and
\cite{nualart-book} for details. Any real, separable Hilbert space
${\mathfrak{H}}$ gives
rise to an isonormal Gaussian process: a centered Gaussian family $%
(G(\varphi ),\varphi \in {\mathfrak{H}})$ of random variables on a
probability space $(\Omega ,\mathcal{F},\mathbb{P})$ such that $\mathbf{E}%
(G(\varphi )G(\psi ))=\langle \varphi ,\psi \rangle
_{{\mathfrak{H}}}$. In this paper, it is enough to use the classical
Wiener space, where $\mathfrak{H}=L^{2}([0,T])$, though any
$\mathfrak{H}$ will also work. In the case
${\mathfrak{H}}=L^{2}([0,T])$, $G$ can be identified with the
stochastic differential of a Wiener process $\{W_t,t\in[0,T]\}$ and
one interprets $G(\varphi ):=\int_{0}^{T}\varphi \left( s\right)
dW\left( s\right) $.

The Wiener chaos of order $p$, denoted by $\mathfrak{H}_p$, is
defined as the closure in $L^{2}\left( \Omega \right) $ of the
linear span of the random variables $H_{p}(G(\varphi ))$, where
$\varphi \in {\mathfrak{H}},\Vert \varphi \Vert _{{\mathfrak{H}}}=1$
and $H_{p}$ is the Hermite polynomial of degree $p$.
 The multiple Wiener stochastic integral
$I_{p}$ with respect to $G\equiv W$, of order $p$ is an isometry
between the Hilbert space ${\mathfrak{H}}^{\odot p}= L^2_{sym}\left(
\left[0,T \right]^p\right)$ (symmetric tensor product) equipped with
the scaled norm $\sqrt{p!}\Vert \cdot \Vert
_{{\mathfrak{H}}^{\otimes p}}$ and the Wiener chaos of order $p$
under $L^{2}\left( \Omega \right) $'s norm, that is, the multiple
Wiener stochastic integral of order $p$:
\[I_p : \left({\mathfrak{H}}^{\odot p},\sqrt{p!}\Vert \cdot \Vert
_{{\mathfrak{H}}^{\otimes p}}\right)\longrightarrow
\left(\mathfrak{H}_p,L^{2}\left( \Omega \right) \right)\]
is a linear isometry defined by $I_p(f^{\otimes p}) = H_p(G(f))$.\\

\noindent $\bullet $ \textbf{Multiple Wiener-It\^o integral.} If
$f\in L^2\left( \left[0,T \right]^p\right)$ is symmetric, we can
also rewrite $I_p( f )$ as the following iterated adapted It\^o
stochastic integral: \begin{eqnarray} I_p(f)&=&\displaystyle
\int_{\left[0,T \right]^p}f(t_1,\ldots ,t_p)dW_{t_1}\ldots
dW_{t_p}\nonumber\\
&=&p!\displaystyle \int_0^TdW_{t_1}\displaystyle
\int_0^{t_1}dW_{t_2}\ldots \displaystyle
\int_0^{t_{p-1}}dW_{t_p}f(t_1,\ldots ,t_p).\label{iterated-integral}
\end{eqnarray}

\noindent $\bullet $ \textbf{The Wiener chaos expansion.} For any
$F\in
L^{2}\left( \Omega \right) $, there exists a unique sequence of functions $%
f_{p}\in {\mathfrak{H}}^{\odot p}$ such that
\begin{equation*}
F=\mathbf{E}[F]+\sum_{p=1}^{\infty }I_{p}(f_{p}),
\end{equation*}%
where the terms are all mutually orthogonal in $L^{2}\left( \Omega
\right) $ and
\begin{equation*}
\mathbf{E}\left[ I_{p}(f_{p})^{2}\right] =p!\Vert f_{p}\Vert
_{{\mathfrak{H}}^{\otimes p}}^{2}.
\end{equation*}

\noindent $\bullet $ \textbf{Product formula and contractions.}  For
any integers $p$, $q \geq 1$ and
symmetric integrands $f\in \mathfrak{H}^{\odot p}$ and $g\in \mathfrak{H}%
^{\odot q}$,
\begin{equation}
I_{p}(f)I_{q}(g)=\sum_{r=0}^{p\wedge q}r!\binom{p}{r} \binom{q}{r}
I_{p+q-2r}(f\widetilde\otimes_r g);  \label{product-formula}
\end{equation}%
where $%
f\otimes _{r}g$  is the contraction of order $r$ of $f$ and $g$
which is an element of ${\mathfrak{H}}^{\otimes (p+q-2r)}$ defined
by
\begin{eqnarray*}
    &&(f\otimes _{r}g)(s_{1},\ldots ,s_{p-r},t_{1},\ldots ,t_{q-r}) \\
    &&\qquad :=\int_{[0,T]^{p+q-2r}}f(s_{1},\ldots ,s_{p-r},u_{1},\ldots
    ,u_{r})g(t_{1},\ldots ,t_{q-r},u_{1},\ldots ,u_{r})\,du_{1}\cdots
    du_{r},
\end{eqnarray*}
 while $f\widetilde\otimes_r g$ denotes its symmetrization. More generally the symmetrization $\tilde{f}$ of a function $f$ is defined by
  $\tilde{f}(x_{1},\ldots,x_{p}) = \frac{1}{p!} \sum\limits_{\sigma} f(x_{\sigma(1)},...,x_{\sigma(p)})$ where the sum runs over all permutations $\sigma$ of $ \{1,...,p\}$.\\

\noindent $\bullet $ \textbf{Kolmogorov distance between random
variables}. If $X,Y$ are two real-valued random variables, then the
Kolmogorov distance  between the law of $X$ and the law of $Y$ is
given by
\begin{equation*}
d_{Kol}\left( X,Y\right):= \sup_{z\in \mathbb{R}}\left\vert
\mathbb{P}\left(X\leq z\right)-\mathbb{P}\left( Y\leq
z\right)\right\vert.
\end{equation*}

\noindent $\bullet $ \textbf{Optimal Berry-Ess\'een bound for the
CLT of  $\boldsymbol{F_t/G_t}$}. Let $Z$ denote the standard normal
law. Recently, using techniques relied on the combination of
Malliavin calculus and Stein's method (see, e.g., \cite{NP-book}),
the following observation  provided lower and upper bounds of the
Kolmogrov distance for the Central Limit Theorem (CLT) of $F_t/G_t$,
where $F_t$ and $G_t$ are functionals of Gaussian fields.

 Fix $T>0$. Let $f_t,g_t\in \mathfrak{H}^{\odot 2}$ for all
$t\in[0,T)$, and let  $b_t$ be a positive function of $t$ such that
$I_2(g_t)+b_t>0$ almost surely for all $t\in[0,T)$.
\begin{itemize}
\item[(a)] If $\max_{i=1,2,3}\psi_i(t)\rightarrow 0$ as $t\uparrow
T$, where for every $t\in[0,T)$,
\begin{eqnarray*}
\psi_1(t)&:=&\frac{1}{b_t^2}\sqrt{\left(b^2_t-2\Vert f_t\Vert^2_{\mathfrak{H}^{\otimes 2}}\right)^2+8\Vert f_t\otimes_1 f_t\Vert^2_{\mathfrak{H}^{\otimes 2}}},\\
\psi_2(t)&:=&\frac{2}{b_t^2}\sqrt{2 \Vert f_t \otimes_1
g_t\Vert_{\mathfrak{H}^{\otimes2}}+\langle
f_t,g_t\rangle^2_{\mathfrak{H}^{\otimes 2}}},\\
\psi_3(t)&:=&\frac{2}{b_t^2}\sqrt{\Vert
g_t\Vert^4_{\mathfrak{H}^{\otimes 2}} +2\Vert g_t\otimes_1
g_t\Vert^2_{\mathfrak{H}^{\otimes 2}}},
\end{eqnarray*}
then (see Corollary 1   in \cite{kp-JVA}), there exists a positive
constant $C$ such that for all $t$ sufficiently near $T$,
\begin{equation}
\sup_{z\in \mathbb{R}}\left\vert \mathbb{P}\left(\frac{I_2(f_t)}{
I_2(g_t)+b_t}\leq z\right)-\mathbb{P}\left( Z\leq
z\right)\right\vert\leq C  \max_{i=1,2,3} \psi_i(t). \label{kp}
\end{equation}

\item[(b)]If, as $t\uparrow T$,
\begin{eqnarray*}
\Vert f_t\otimes_1 f_t\Vert_{\mathfrak{H}^{\otimes
2}}\longrightarrow0,\qquad
 \frac{2\Vert f_t\Vert^2_{\mathfrak{H}^{\otimes
2}}-1}{\langle f_t\otimes_1 f_t,f_t\rangle_{\mathfrak{H}^{\otimes
2}}}\longrightarrow0,\ \mbox{ and }\ \frac{\Vert f_t\otimes_1
f_t\Vert_{\mathfrak{H}^{\otimes 2}}}{\langle f_t\otimes_1
f_t,f_t\rangle_{\mathfrak{H}^{\otimes 2}}}\longrightarrow \rho\neq0,
\end{eqnarray*}
 then  (see Corollary 2   in \cite{kp-JVA}), there exists a positive constant $c$ such that for all $t$
sufficiently near $T$,
\begin{equation}
\sup_{z\in \mathbb{R}}\left\vert \mathbb{P}\left(\frac{I_2(f_t)}{
I_2(g_t)+b_t}\leq z\right)-\mathbb{P}\left( Z\leq
z\right)\right\vert\geq c \vert \langle f_t\otimes_1
f_t,f_t\rangle_{\mathfrak{H}^{\otimes 2}}\vert. \label{kp2}
\end{equation}
\end{itemize}

Throughout the paper $Z$ denotes a standard normal random variable.
Also,  $ C$ denotes a generic positive constant (perhaps depending
on $\alpha$ and $T$, but not on anything else), which may change
from line to line.

\section{Optimal rate of convergence of the MLE}
In this section we consider the problem of optimal rate of
convergence to normality of the MLE $\widetilde{\alpha}_t $ given in
(\ref{MLE}). More precisely,  we want to provide an optimal
Berry-Ess\'een bound in the Kolmogorov distance for
$\widetilde{\alpha}_t$. In what follows we suppose that
$\alpha>\frac12$.

To proceed, let us start with useful notations  needed in what
follows. Because (\ref{bridge}) is linear, it is immediate to solve
it explicitly; one then gets the following formula: \ben X_t=\left(
T-t\right)^{\alpha} \displaystyle \int_0^t\left(
T-s\right)^{-\alpha} dW_s,~~~0\leq t< T.\label{SOL} \een
 Define for every $t\in[0,T)$,
 \begin{eqnarray}\lambda_t:=\frac{\vert\log\left(T- t
\right)\vert}{2\alpha-1}.\label{lambda}\end{eqnarray}
 It follows  from  \eqref{iterated-integral} and \eqref{SOL} that for every $t\in[0,T)$,
\begin{eqnarray} \sqrt{\frac{2\alpha-1}{\vert\log\left(T- t
\right)\vert}}\displaystyle
\int_0^t\frac{X_s}{T-s}dW_s&=&\frac{1}{\sqrt{\lambda_t}}\displaystyle
\int_0^t\frac{X_s}{T-s}dW_s\nonumber\\&=&\frac{1}{\sqrt{\lambda_t}}\displaystyle
\int_0^t\displaystyle \int_0^s\left( T-s\right)^{\alpha-1} \left(
T-r\right)^{-\alpha} dW_rdW_s\nonumber\\&=&
\frac{1}{2\sqrt{\lambda_t}}\displaystyle \int_0^t\displaystyle
\int_0^t\left( T-s\vee r\right)^{\alpha-1} \left( T-s\wedge
r\right)^{-\alpha} dW_rdW_s\nonumber\\&=:&
I_2\left(f_t\right),\label{numerator}
\end{eqnarray}
where $f_t$ is a symmetric function defined by
\begin{eqnarray}
    f_t(u,v)=\frac{1}{2\sqrt{\lambda_t}}\left(T-u \vee v \right)^{\alpha-1}
    \left(T-u \wedge v \right)^{-\alpha} \mathds{1}_{\left[ 0,t\right]^2 }
    (u,v).\label{kernel-f}
    \end{eqnarray}
On the other hand, using \eqref{product-formula} and \eqref{SOL}, we
can write
\begin{eqnarray*}
X_s^2&=&\left[I_1\left(\left( T-s\right)^{\alpha}\left(
T-u\right)^{-\alpha}\mathds{1}_{\left[ 0,s\right] }
(u)\right)\right]^2\\
&=&I_2\left(\left( T-s\right)^{2\alpha}\left(
T-u\right)^{-\alpha}\left( T-v\right)^{-\alpha}\mathds{1}_{\left[
0,s\right]^2 } (u,v)\right)+ \int_0^s\left(
T-s\right)^{2\alpha}\left( T-u\right)^{-2\alpha}du.
\end{eqnarray*}
Hence, we have
\begin{eqnarray}  \frac{1}{\lambda_t}\int_0^t\frac{X^2_s}{\left(
T-s\right) ^2}ds  = I_2(g_t)+b_t,\label{denominator}
\end{eqnarray}
where
\begin{eqnarray}
g_t(u,v)&=& \frac{1}{\lambda_t}\int_0^t\left(
T-s\right)^{2\alpha-2}\left( T-u\right)^{-\alpha}\left(
T-v\right)^{-\alpha} \mathds{1}_{\left[ 0,s\right]^2 }(u,v)
ds\nonumber\\&=& \frac{1}{\lambda_t}\left(
T-u\right)^{-\alpha}\left( T-v\right)^{-\alpha}\mathds{1}_{\left[
0,t\right]^2 }(u,v)\int_{u \vee v}^t\left( T-s\right)^{2\alpha-2}
 ds\nonumber\\&=&
\frac{\left( T-u\right)^{-\alpha}\left(
T-v\right)^{-\alpha}}{\vert\log\left(T- t \right)\vert}\left[\left(
T-u \vee v\right)^{2\alpha-1}-\left( T-t\right)^{2\alpha-1}\right]
\mathds{1}_{\left[ 0,t\right]^2 }(u,v),\label{kernel-g}
\end{eqnarray}
and
\begin{eqnarray}
b_t&=& \frac{1}{\lambda_t}\int_0^t\int_0^s\left(
T-s\right)^{2\alpha-2}\left(
T-u\right)^{-2\alpha}duds\nonumber\\
&=&\frac{1}{\vert\log\left(T- t \right)\vert} \int_0^t \left(
T-s\right)^{-1}-T^{1-2\alpha}\left(
T-s\right)^{2\alpha-2}ds\nonumber\\
&=&1+\frac{\log(T)}{\vert\log\left(T- t
\right)\vert}-\frac{1}{(2\alpha-1)\vert\log\left(T- t \right)\vert}
\left(1-  \left(
\frac{T-t}{T}\right)^{2\alpha-1}\right).\label{exp-b}
\end{eqnarray}
Therefore, combining (\ref{MLE2}), (\ref{numerator}) and
(\ref{denominator}), we can write
 \begin{eqnarray}
\sqrt{\lambda_t}\left(\alpha-\widetilde{\alpha}_t\right)=
\frac{I_2(f_t)}{I_2(g_t)+b_t}, \label{MLE-fraction}\end{eqnarray}
where $\lambda_t$, $f_t$, $g_t$ and $b_t$ are given in
(\ref{lambda}), (\ref{kernel-f}), (\ref{kernel-g}) and
(\ref{exp-b}), respectively.\\

In order to prove our main result we make use of the following
technical lemmas.
\begin{lemma} Suppose that $\alpha>\frac12$.
Let  $\lambda_t$, $f_t$ and $b_t$  be the functions  given by
(\ref{lambda}), (\ref{kernel-f}) and (\ref{exp-b}), respectively.
Then  we have,  as $t\uparrow T$,
\begin{eqnarray}
2\Vert f_t\Vert_{\mathfrak{H}^{\otimes2}}^2-1&=&
\frac{\log(T)}{\left( 2\alpha-1\right)\lambda_t}-\frac{1}{\left(
2\alpha-1\right)^2\lambda_t} +
o\left(\frac{1}{\lambda_t}\right),\label{equ-f1}
\end{eqnarray}
\begin{eqnarray}
b_t^2-1&=& \frac{2\log(T)}{\left(
2\alpha-1\right)\lambda_t}-\frac{2}{\left(
2\alpha-1\right)^2\lambda_t} +
o\left(\frac{1}{\lambda_t}\right),\label{equ-f2}
\end{eqnarray}
\begin{eqnarray}
\langle f_t\otimes_1 f_t,f_t\rangle_{\mathfrak{H}^{\otimes
2}}=\frac{3}{4(2\alpha-1)\sqrt{\lambda_t}}+o\left(\frac{1}{\sqrt{\lambda_t}}\right),\label{equ-f3}
\end{eqnarray}
\begin{eqnarray}
\Vert  f_t \otimes_{1}   f_t \Vert_{\mathfrak{H}^{\otimes2}}^2 &=&
\frac{5}{4(2\alpha-1)^2\lambda_t} +
o\left(\frac{1}{\lambda_t}\right),\label{equ-f4}
\end{eqnarray}
where the notation $o(1/\lambda_t^{\beta})$ means that
$\lambda_t^{\beta}o(1/\lambda_t^{\beta})\longrightarrow0$ as
$t\uparrow T$.
\end{lemma}
\begin{proof}
Suppose that $0<T-t<1$. So, $-\log(T-t)=\vert\log(T-t)\vert$. We
also  notice that  $\lambda_t\longrightarrow\infty$ as $t\uparrow
T$. Since the function $f_t$ is symmetric, we have
\begin{eqnarray*}
\Vert f_t\Vert_{\mathfrak{H}^{\otimes2}}^2&=&\frac{1}{4\lambda_t}
\int_{\left[0,t \right] ^2}\left(T-x \vee y
\right)^{2\alpha-2}\left(T-x \wedge y \right)^{-2\alpha}dxdy
\\&=& \frac{1}{2\lambda_t}\int_0^tdy\displaystyle \int_0^y\left(T-
y \right)^{2\alpha-2}\left(T-x\right)^{-2\alpha}dx \\&=&
 \frac{1}{2(2\alpha-1)\lambda_t}
\int_0^t\left(\left(T- y \right)^{-1}- T^{-2\alpha+1}\left(T- y
\right)^{2\alpha-2}\right)dy \\&=& \frac{1}{2\vert\log\left(T- t
\right)\vert}\left( \log(T)-\log\left(T- t \right)+\frac{\left(
(T-t)/T\right)^{2\alpha-1} }{\left( 2\alpha-1\right)
}-\frac{1}{\left( 2\alpha-1\right) }\right)
\\&=&
\frac{1}{2}\left(1+ \frac{\log(T)}{\vert\log\left(T- t
\right)\vert}+\frac{\left( (T-t)/T\right)^{2\alpha-1} }{\left(
2\alpha-1\right)\vert\log\left(T- t \right)\vert}-\frac{1}{\left(
2\alpha-1\right)\vert\log\left(T- t \right)\vert}\right)\\&=&
\frac{1}{2}\left(1+  \frac{\log(T)}{\left(
2\alpha-1\right)\lambda_t}-\frac{1}{\left(
2\alpha-1\right)^2\lambda_t}+\frac{\left(T-t\right)^{2\alpha-1}
}{\left( 2\alpha-1\right)^2T^{2\alpha-1}\lambda_t\vert}\right)\\&=&
\frac{1}{2}\left(1+ \frac{\log(T)}{\left(
2\alpha-1\right)\lambda_t}-\frac{1}{\left(
2\alpha-1\right)^2\lambda_t} +
o\left(\frac{1}{\lambda_t}\right)\right)
\end{eqnarray*}
as $t\uparrow T$, where the latter equality comes from the fact that $\alpha>\frac12$.\\
Thus, we can deduce
\begin{eqnarray*}
2\Vert f_t\Vert_{\mathfrak{H}^{\otimes2}}^2-1&=&
\frac{\log(T)}{\left( 2\alpha-1\right)\lambda_t}-\frac{1}{\left(
2\alpha-1\right)^2\lambda_t} + o\left(\frac{1}{\lambda_t}\right),
\end{eqnarray*}
which proves (\ref{equ-f1}). On the other hand, the estimate
(\ref{equ-f2}) is a direct consequence of (\ref{exp-b}).\\
Let us prove (\ref{equ-f3}), we have
\begin{eqnarray*}
\langle f_t\otimes_1 f_t,f_t\rangle_{\mathfrak{H}^{\otimes 2}}&=&
\int_{[0,t]^3}f_t(x_1,x_2)f_t(x_2,x_3)f_t(x_3,x_1)dx_1dx_2dx_3\\
&=&3! \int_0^t dx_3\int_0^{x_3} dx_2\int_0^{x_2} dx_1
f_t(x_1,x_2)f_t(x_2,x_3)f_t(x_3,x_1), \end{eqnarray*} where we used
the fact that the integrand   is symmetric. \\
Hence, using \eqref{kernel-f}, we get
\begin{eqnarray*}
&&\langle f_t\otimes_1 f_t,f_t\rangle_{\mathfrak{H}^{\otimes
2}}\\&=&\frac{3}{4(\lambda_t)^{3/2}}\int_0^t dx_3\int_0^{x_3}
dx_2\int_0^{x_2} dx_1\left(T-x_3 \right)^{2\alpha-2}
    \left(T-x_2\right)^{-1}\left(T-x_1 \right)^{-2\alpha}
    \\
&=&\frac{3}{4(2\alpha-1)(\lambda_t)^{3/2}}\int_0^t dx_3\int_0^{x_3}
dx_2 \left(T-x_3 \right)^{2\alpha-2}
    \left(T-x_2\right)^{-1}\left[\left(T-x_2
    \right)^{1-2\alpha}-T^{1-2\alpha}\right]
     \\
&=&\frac{3}{4(2\alpha-1)(\lambda_t)^{3/2}}\int_0^t dx_3  \left(T-x_3
\right)^{2\alpha-2}\left[
   \frac{\left(T-x_3\right)^{1-2\alpha}-T^{1-2\alpha}}{2\alpha-1}+T^{1-2\alpha}
\log\left(\frac{T-x_3}{T}\right)\right]
\\
&=&\frac{3}{4(2\alpha-1)^2(\lambda_t)^{3/2}}\left[\log(T)-\log(T-t)+T^{1-2\alpha}\frac{\left(T-t\right)^{2\alpha-1}-T^{2\alpha-1}}{2\alpha-1}
\right.\\&&\left.
-T^{1-2\alpha}\log\left(\frac{T-t}{T}\right)\left(T-t\right)^{2\alpha-1}
+T^{1-2\alpha}\left(\left(T-t\right)^{2\alpha-1}-T^{2\alpha-1}\right)
\right]
\\
&=&\frac{3}{4(2\alpha-1)\sqrt{\lambda_t}}+o\left(\frac{1}{\sqrt{\lambda_t}}\right),
\end{eqnarray*}
which proves (\ref{equ-f3}).\\
 Now let us prove (\ref{equ-f4}). By  \eqref{kernel-f}, we obtain
\begin{eqnarray}
&&\Vert  f_t \otimes_{1}   f_t
\Vert_{\mathfrak{H}^{\otimes2}}^2\nonumber
\\&=& \int_{\left[0,t \right] ^2}\left(\int_{0}^t f_t(x_1,x_2)
f_t(x_3,x_2)dx_2\right) ^2dx_1dx_3\nonumber
 \\&=& \int_{\left[0,t \right] ^4}  f_t(x_1,x_2) f_t(x_2,x_3) f_t(x_3,x_4) f_t(x_4,x_1)dx_1dx_2dx_3dx_4\nonumber
\\&=& \frac{16}{16\lambda_t^2} \int_{\{0<x_1<x_2<x_3<x_4<t\}} \left(T-x_4\right)^{2\alpha-2}\left(T-x_3\right)^{-1}\left(T-x_2\right)^{-1}
\left(T-x_1\right)^{-2\alpha}dx_1dx_2dx_3dx_4\nonumber
\\&&+\frac{8}{16\lambda_t^2}  \int_{\{0<x_1< x_3< x_2< x_4<t\}} \left(T-x_4\right)^{2\alpha-2}\left(T-x_2\right)^{2\alpha-2}
\left(T-x_3\right)^{-2\alpha}\left(T-x_1\right)^{-2\alpha}dx_1dx_3dx_2dx_4\nonumber
\\&=:&   A_{t,1}+A_{t,2}.\label{A1+A2}
\end{eqnarray}
For the term $ A_{t,1} $, we have
\begin{eqnarray}A_{t,1}&=&\frac{1}{\lambda_t^2}
\int_0^tdx_4  \int_0^{x_4}dx_3  \int_0^{x_3}dx_2  \int_0^{x_2}dx_1
\left(T-x_4\right)^{2\alpha-2}\left(T-x_3\right)^{-1}\left(T-x_2\right)^{-1}\left(T-x_1\right)^{-2\alpha}\nonumber
\\&=&
\frac{1}{(2\alpha-1)\lambda_t^2} \int_0^tdx_4  \int_0^{x_4}dx_3
\int_0^{x_3}dx_2
\left(T-x_4\right)^{2\alpha-2}\left(T-x_3\right)^{-1}\nonumber\\&&\qquad\qquad\qquad\qquad\times\left[\left(T-x_2\right)^{-2\alpha}-T^{1-2\alpha}\left(T-x_2\right)^{-1}\right]\nonumber
\\&=&
\frac{1}{(2\alpha-1)\lambda_t^2} \int_0^tdx_4  \int_0^{x_4}dx_3
\left(T-x_4\right)^{2\alpha-2}\left[\frac{\left(T-x_3\right)^{-2\alpha}-T^{1-2\alpha}\left(T-x_3\right)^{-1}}{2\alpha-1}
\right.\nonumber\\&&\left.\qquad\qquad\qquad+T^{1-2\alpha}\log\left(\frac{T-x_3}{T}\right)\left(T-x_3\right)^{-1}\right]
\nonumber\\&=& \frac{1}{(2\alpha-1)\lambda_t^2} \int_0^tdx_4\left[
\frac{\left(T-x_4\right)^{-1}-T^{1-2\alpha}(T-x_4)^{2\alpha-2}}{(2\alpha-1)^2}\right.\nonumber\\
&&\left.\quad
+\frac{T^{1-2\alpha}}{2\alpha-1}\log\left(\frac{T-x_4}{T}\right)
\left(T-x_4\right)^{2\alpha-2}
-\frac{T^{1-2\alpha}}{2}\log^2\left(\frac{T-x_4}{T}\right)\left(T-x_4\right)^{2\alpha-2}\right]\nonumber
\\&=&
\frac{1}{(2\alpha-1)^3\lambda_t^2}
\int_0^tdx_4\left(T-x_4\right)^{-1}+R_t\nonumber\\&=&
\frac{1}{(2\alpha-1)^2\lambda_t} +
o\left(\frac{1}{\lambda_t}\right),\label{A1}
\end{eqnarray}
where, using integration by parts, straightforward calculations lead
to
\begin{eqnarray*}R_t&:=&
\frac{1}{(2\alpha-1)\lambda_t^2}
\int_0^tdx_4\left[\frac{-T^{1-2\alpha}(T-x_4)^{2\alpha-2}}{(2\alpha-1)^2}+
\frac{T^{1-2\alpha}}{2\alpha-1}\log\left(\frac{T-x_4}{T}\right)
\left(T-x_4\right)^{2\alpha-2}\right.
\\&&\left.\qquad\qquad-\frac{T^{1-2\alpha}}{2}\log^2\left(\frac{T-x_4}{T}\right)\left(T-x_4\right)^{2\alpha-2}\right]\\
&=& o\left(\frac{1}{\lambda_t}\right).
\end{eqnarray*}
Similarly, we obtain
\begin{eqnarray}A_{t,2}&=&\frac{1}{2\lambda_t^2} \int_0^tdx_4  \int_0^{x_4}dx_2  \int_0^{x_2}dx_3  \int_0^{x_3}dx_1
\left(T-x_4\right)^{2\alpha-2}\left(T-x_2\right)^{2\alpha-2}
\left(T-x_3\right)^{-2\alpha}\left(T-x_1\right)^{-2\alpha}\nonumber\\
&=&\frac{1}{2(2\alpha-1)\lambda_t^2} \int_0^tdx_4  \int_0^{x_4}dx_2
\int_0^{x_2}dx_3
\left(T-x_4\right)^{2\alpha-2}\left(T-x_2\right)^{2\alpha-2}\nonumber\\&&
\qquad\qquad\qquad\qquad\qquad\qquad\times\left[\left(T-x_3\right)^{1-4\alpha}-T^{1-2\alpha}\left(T-x_3\right)^{-2\alpha}\right]\nonumber
\\
&=&\frac{1}{4(2\alpha-1)^2\lambda_t^2} \int_0^tdx_4 \int_0^{x_4}dx_2
 \left(T-x_4\right)^{2\alpha-2}\left[\left(T-x_2\right)^{-2\alpha}-T^{2-4\alpha}\left(T-x_2\right)^{2\alpha-2}\right.\nonumber\\&&
\qquad\qquad\qquad\qquad\left.
 -2T^{1-2\alpha}\left(\left(T-x_2\right)^{-1}-T^{1-2\alpha}\left(T-x_2\right)^{2\alpha-2}\right)\right]\nonumber
 \\
&=&\frac{1}{4(2\alpha-1)^2\lambda_t^2} \int_0^tdx_4 \int_0^{x_4}dx_2
 \left(T-x_4\right)^{2\alpha-2}\left[\left(T-x_2\right)^{-2\alpha}+T^{2-4\alpha}\left(T-x_2\right)^{2\alpha-2}\right.\nonumber\\&&
\qquad\qquad\qquad\qquad\qquad\left.
 -2T^{1-2\alpha}\left(T-x_2\right)^{-1}\right]\nonumber
 \\&=&\frac{1}{4(2\alpha-1)^3\lambda_t^2} \int_0^t
 \left(T-x_4\right)^{-1}dx_4+S_t\nonumber
 \\&=&\frac{1}{4(2\alpha-1)^2\lambda_t}
+ o\left(\frac{1}{\lambda_t}\right),\label{A2}
\end{eqnarray}
where, by integration by parts, it is easy to check that
\begin{eqnarray*}S_t &:=&\frac{-T^{1-2\alpha}}{4(2\alpha-1)^3\lambda_t^2} \int_0^t\left[\left(T-x_4\right)^{2\alpha-2}
-T^{1-2\alpha}\left(\left(T-x_4\right)^{4\alpha-3}+T^{2\alpha-1}\left(T-x_4\right)^{2\alpha-2}\right)\right]dx_4\\
&&+\frac{T^{1-2\alpha}}{2(2\alpha-1)^3\lambda_t^2}
\int_0^t\left(T-x_4\right)^{2\alpha-2}\log\left(\frac{T-x_4}{T}\right)dx_4
 \\&=& o\left(\frac{1}{\lambda_t}\right).
\end{eqnarray*}
Combining (\ref{A1+A2}), (\ref{A1}) and (\ref{A2}), we obtain
(\ref{equ-f4}), and therefore, the proof is complete.
\end{proof}

\begin{lemma} Suppose that $\alpha>\frac12$. Let  $\lambda_t$, $f_t$ and $g_t$    be the functions  given
by (\ref{lambda}), (\ref{kernel-f}) and (\ref{kernel-g}),
respectively. Then, for all $(T-1/e)\vee0<t<T$,
\begin{eqnarray}
\Vert
g_t\Vert_{\mathfrak{H}^{\otimes2}}&\leq&\frac{C}{\sqrt{\lambda_t}},\label{equ-g1}
\end{eqnarray}
\begin{eqnarray}\Vert  g_t \otimes_{1}   g_t \Vert_{\mathfrak{H}^{\otimes2}}&\leq&
\frac{C}{ \lambda_t^{3/2}},\label{equ-g2}
\end{eqnarray}
\begin{eqnarray} \vert\langle f_t,g_t\rangle_{\mathfrak{H}^{\otimes
2}}\vert\ &\leq&\frac{C}{\sqrt{\lambda_t}},\label{equ-g3}
\end{eqnarray}
\begin{eqnarray}
\Vert  f_t \otimes_{1}   g_t \Vert_{\mathfrak{H}^{\otimes2}} &\leq&
\frac{C}{ \lambda_t}.\label{equ-g4}
\end{eqnarray}
\end{lemma}

\begin{proof}Note that if  $(T-1/e)\vee0<t<T$,
$-\log(T-t)=\vert\log(T-t)\vert$ and $\vert\log(T-t)\vert>1$. From
\eqref{kernel-g} we have
\begin{eqnarray*}
\Vert
g_t\Vert_{\mathfrak{H}^{\otimes2}}^2&=&\int_0^t\int_0^t\frac{\left(
T-u\right)^{-2\alpha}\left( T-v\right)^{-2\alpha}}{\vert\log\left(T-
t \right)\vert^2}\left[\left( T-u \vee v\right)^{2\alpha-1}-\left(
T-t\right)^{2\alpha-1}\right]^2dudv\\
&\leq&\int_0^t\int_0^t\frac{\left( T-u\right)^{-2\alpha}\left(
T-v\right)^{-2\alpha}}{\vert\log\left(T- t \right)\vert^2} \left(
T-u \vee v\right)^{4\alpha-2}dudv\\
&=&2\int_0^tdv\int_0^vdu\frac{\left( T-u\right)^{-2\alpha}\left(
T-v\right)^{2\alpha-2}}{\vert\log\left(T- t \right)\vert^2}
\\
&=&2\int_0^tdv \frac{\left(
T-v\right)^{2\alpha-2}}{(2\alpha-1)\vert\log\left(T- t
\right)\vert^2}\left[\left(
T-v\right)^{1-2\alpha}-T^{1-2\alpha}\right]
\\
&\leq&2\int_0^tdv \frac{\left(
T-v\right)^{-1}}{(2\alpha-1)\vert\log\left(T- t \right)\vert^2}
\\
&=&2 \frac{\log\left(T\right)-\log\left(T- t
\right)}{(2\alpha-1)\vert\log\left(T- t \right)\vert^2}
\\
&\leq&\frac{C}{\lambda_t},
\end{eqnarray*}
which proves (\ref{equ-g1}). \\
Now let us prove (\ref{equ-g2}).
Using \eqref{kernel-g} and the fact that $\alpha>\frac12$, we get
\begin{eqnarray}
0\leq g_t(u,v)&=& \frac{\left( T-u\right)^{-\alpha}\left(
T-v\right)^{-\alpha}}{\vert\log\left(T- t \right)\vert}\left[\left(
T-u \vee v\right)^{2\alpha-1}-\left( T-t\right)^{2\alpha-1}\right]
\mathds{1}_{\left[ 0,t\right]^2 }(u,v)\nonumber\\
&\leq&\frac{\left( T-u\right)^{-\alpha}\left(
T-v\right)^{-\alpha}}{\vert\log\left(T- t \right)\vert}\left( T-u
\vee v\right)^{2\alpha-1}
\mathds{1}_{\left[ 0,t\right]^2 }(u,v)\nonumber\\
&=:&h_t(u,v).\label{defi-h}
\end{eqnarray}
Further, notice that
\begin{eqnarray}
&& \int_{\left[0,t \right] ^4}  h_t(x_1,x_2) h_t(x_2,x_3)
h_t(x_3,x_4) h_t(x_4,x_1)dx_1dx_2dx_3dx_4\nonumber
\\&=&\frac{16}{\vert\log\left(T- t \right)\vert^4}\int_{\{0<x_1<x_2<x_3<x_4<t\}} \left(
T-x_4\right)^{2\alpha-2}\left( T-x_3\right)^{-1}\left(
T-x_2\right)^{-1}\nonumber\\&&\qquad\qquad\qquad\qquad\qquad\qquad\qquad\times \left( T-x_1\right)^{-2\alpha}dx_1dx_2dx_3dx_4\nonumber\\
&&+\frac{8}{\vert\log\left(T- t \right)\vert^4}\int_{\{0<x_1< x_3<
x_2< x_4<t\}} \left( T-x_4\right)^{2\alpha-2}\left(
T-x_2\right)^{2\alpha-2}\left(
T-x_3\right)^{-2\alpha}\nonumber\\&&\qquad\qquad\qquad\qquad\qquad\qquad\qquad\times\left(
T-x_1\right)^{-2\alpha}dx_1dx_3dx_2dx_4\nonumber
\\&=&  \frac{16}{(2\alpha-1)^4\lambda_t^2}
A_{t,1}+\frac{16}{(2\alpha-1)^4\lambda_t^2}A_{t,2}\nonumber
\\&\leq& \frac{C}{
\lambda_t^3},\label{upper-h}
\end{eqnarray}
where the latter inequality comes from (\ref{A1}) and (\ref{A2}).\\
Thus, combining \eqref{defi-h} and \eqref{upper-h}, we obtain
\begin{eqnarray*}
\Vert  g_t \otimes_{1}   g_t
\Vert_{\mathfrak{H}^{\otimes2}}^2\nonumber &=& \int_{\left[0,t
\right] ^2}\left(\int_{0}^t g_t(x_1,x_2) g_t(x_3,x_2)dx_2\right)
^2dx_1dx_3\nonumber
 \\&=& \int_{\left[0,t \right] ^4}  g_t(x_1,x_2) g_t(x_2,x_3) g_t(x_3,x_4) g_t(x_4,x_1)dx_1dx_2dx_3dx_4\nonumber
\\&\leq&\int_{\left[0,t \right] ^4}  h_t(x_1,x_2) h_t(x_2,x_3) h_t(x_3,x_4)
h_t(x_4,x_1)dx_1dx_2dx_3dx_4\\&\leq& \frac{C}{ \lambda_t^3},
\end{eqnarray*}
which implies (\ref{equ-g2}). For (\ref{equ-g3}), since $f_t$ and
$g_t$ are symmetric, we have
\begin{eqnarray*}
\vert\langle f_t,g_t\rangle_{\mathfrak{H}^{\otimes
2}}\vert&=&\int_0^t \int_0^t f_t(u,v)g_t(u,v)dudv
    \\
&=&\frac{1}{(2\alpha-1)(\lambda_t)^{3/2}}\int_0^t dv\int_0^v du
\left(T-v \right)^{-1}
    \left(T-u\right)^{-2\alpha}\left[\left(T-v
    \right)^{2\alpha-1}-\left(T-t\right)^{2\alpha-1}\right]
    \\
&\leq&\frac{1}{(2\alpha-1)(\lambda_t)^{3/2}}\int_0^t dv\left(T-v
\right)^{2\alpha-2}\int_0^v du
    \left(T-u\right)^{-2\alpha}
     \\
&\leq&\frac{1}{(2\alpha-1)^2(\lambda_t)^{3/2}}\int_0^t dv\left(T-v
\right)^{-1}\\
&\leq&\frac{C}{\sqrt{\lambda_t}}.
\end{eqnarray*}
 To finish the proof it remains to prove  the estimate (\ref{equ-g4}). It follows from \eqref{defi-h}
 that
\begin{eqnarray*}
\Vert  f_t \otimes_{1}   g_t \Vert_{\mathfrak{H}^{\otimes2}}^2 &=&
\int_{\left[0,t \right] ^2}\left(\int_{0}^t f_t(x_1,x_2)
g_t(x_3,x_2)dx_2\right) ^2dx_1dx_3
 \\&=& \int_{\left[0,t \right] ^4}  f_t(x_1,x_2) g_t(x_2,x_3) g_t(x_3,x_4) f_t(x_4,x_1)dx_1dx_2dx_3dx_4
 \\&\leq& \int_{\left[0,t \right] ^4}  f_t(x_1,x_2) h_t(x_2,x_3) h_t(x_3,x_4) f_t(x_4,x_1)dx_1dx_2dx_3dx_4
\end{eqnarray*}
Moreover, we notice that
\begin{eqnarray*}&&\int_{\left[0,t \right] ^4}  f_t(x_1,x_2) h_t(x_2,x_3) h_t(x_3,x_4) f_t(x_4,x_1)dx_1dx_2dx_3dx_4\\
&=& \frac{16}{4(2\alpha-1)^2\lambda_t^3}
\int_{\{0<x_1<x_2<x_3<x_4<t\}}
\left(T-x_1\right)^{-2\alpha}\left(T-x_2\right)^{-1}\left(T-x_3\right)^{-1}\\
&& \qquad\qquad\qquad\qquad\qquad\qquad\qquad
\times\left(T-x_4\right)^{2\alpha-2}dx_1dx_2dx_3dx_4
\\&&+\frac{8}{4(2\alpha-1)^2\lambda_t^3}
\int_{\{0<x_1< x_3< x_2< x_4<t\}}
\left(T-x_1\right)^{-2\alpha}\left(T-x_3\right)^{-2\alpha}
\left(T-x_2\right)^{2\alpha-2}\\
&& \qquad\qquad\qquad\qquad\qquad\qquad\qquad
\times\left(T-x_4\right)^{2\alpha-2}dx_1dx_3dx_2dx_4
\\&=&   \frac{4}{(2\alpha-1)^2\lambda_t}A_{t,1}+\frac{4}{(2\alpha-1)^2\lambda_t}A_{t,2}
\\&\leq& \frac{C}{
\lambda_t^2},
\end{eqnarray*}
where the latter inequality follows from (\ref{A1}) and (\ref{A2}).
\end{proof}

  Now we are ready to state and prove our main result. In the next Theorem we give an explicit optimal bound for the Kolmogorov distance,
   between the law of $ \sqrt{\frac{\vert\log\left( T-t\right)\vert}{2\alpha-1}}\left( \alpha-\widetilde{\alpha}_t\right) $
    and the standard normal law.
\begin{theorem}
Let $ T>0 $, $\alpha>1/2$, and let $ \widetilde{\alpha}_t$ be the
MLE given in (\ref{MLE}).Then there exist constants $0 < c < C <
\infty$, depending only on $\alpha$ and $T$, such that for all $t$
sufficiently near $T$,
\begin{eqnarray*}
\frac{c}{\sqrt{\vert\log\left( T-t\right)\vert}}\leq\sup_{z\in
\mathbb{R}}\left\vert \mathbb{P}\left(\sqrt{\frac{\vert\log\left(
T-t\right)\vert}{2\alpha-1}}\left(
\alpha-\widetilde{\alpha}_t\right)\leq z\right)-\mathbb{P}\left(
Z\leq z\right)\right\vert\leq \frac{C}{\sqrt{\vert\log\left(
T-t\right)\vert}}.
\end{eqnarray*}
\end{theorem}

\begin{proof}According to (\ref{MLE-fraction}) we have
 \begin{eqnarray*}
\sqrt{\frac{\vert\log\left( T-t\right)\vert}{2\alpha-1}}\left(
\alpha-\widetilde{\alpha}_t\right)= \frac{I_2(f_t)}{I_2(g_t)+b_t},
\end{eqnarray*} where $f_t$,  $g_t$ and $b_t$ are given by
(\ref{kernel-f}), (\ref{kernel-g}) and (\ref{exp-b}), respectively.
Let us first show that an upper bound in Kolmogorov distance for a
normal approximation of MLE $\widetilde{\alpha}_t$ is given by the
rate $\frac{1}{\sqrt{\vert\log\left( T-t\right)\vert}}$. Applying
\eqref{kp}, it suffices to prove that
 \begin{eqnarray}\max_{i=1,2,3}
\psi_i(t)\leq \frac{C}{\sqrt{\vert\log\left(
T-t\right)\vert}}.\label{upper-thm}
 \end{eqnarray}
 Using (\ref{equ-f1}), (\ref{equ-f2}) and (\ref{equ-f4}),
we obtain $ \psi_1(t)\leq \frac{C}{\sqrt{\lambda_t}}\leq
\frac{C}{\sqrt{\vert\log\left( T-t\right)\vert}}$.   On the other
hand, by combining (\ref{equ-g3})  and (\ref{equ-g4}), we get $
\psi_2(t)\leq \frac{C}{\sqrt{\lambda_t}}\leq
\frac{C}{\sqrt{\vert\log\left( T-t\right)\vert}}$. Further, the
estimates (\ref{equ-g1})  and (\ref{equ-g2}) imply that $
\psi_3(t)\leq \frac{C}{\lambda_t^{3/4}}\leq
\frac{C}{\sqrt{\vert\log\left( T-t\right)\vert}}$. Therefore,
(\ref{upper-thm}) is obtained. \\
For the lower bound, combining  (\ref{equ-f1}), (\ref{equ-f3}) and
(\ref{equ-f4}) together with the fact that $2\alpha-1>0$ and
$\lambda_t\longrightarrow\infty$ as $t\uparrow T$, we obtain, as
$t\uparrow T$,
\begin{eqnarray*}
\Vert f_t\otimes_1 f_t\Vert_{\mathfrak{H}^{\otimes
2}}=\frac{\sqrt{5}}{2(2\alpha-1)\sqrt{\lambda_t}} +
o\left(\frac{1}{\sqrt{\lambda_t}}\right)\longrightarrow0,
\end{eqnarray*}
\begin{eqnarray*}
 \frac{2\Vert f_t\Vert^2_{\mathfrak{H}^{\otimes
2}}-1}{\langle f_t\otimes_1 f_t,f_t\rangle_{\mathfrak{H}^{\otimes
2}}}=\frac{\frac{\log(T)}{\left(
2\alpha-1\right)\sqrt{\lambda_t}}-\frac{1}{\left(
2\alpha-1\right)^2\sqrt{\lambda_t}} +
o\left(\frac{1}{\sqrt{\lambda_t}}\right)}{\frac{3}{4(2\alpha-1)}+o\left(1\right)}\longrightarrow0,
\end{eqnarray*}
\begin{eqnarray*} \frac{\Vert f_t\otimes_1
f_t\Vert_{\mathfrak{H}^{\otimes 2}}}{\langle f_t\otimes_1
f_t,f_t\rangle_{\mathfrak{H}^{\otimes
2}}}=\frac{\frac{\sqrt{5}}{2}+o\left(1\right)}{\frac34+o\left(1\right)}\longrightarrow
\frac{2\sqrt{5}}{3}\neq0.
\end{eqnarray*}
Moreover, using (\ref{equ-f3}), there is $ c> 0$, depending only on
$\alpha$ and $T$, such that for all $t$ sufficiently near $T$,
\begin{eqnarray*}
\vert\langle f_t\otimes_1 f_t,f_t\rangle_{\mathfrak{H}^{\otimes
2}}\vert=\frac{1}{\sqrt{\lambda_t}}\left\vert\frac{3}{4(2\alpha-1)}+o\left(1\right)\right\vert\geq
\frac{c}{\sqrt{\vert\log\left( T-t\right)\vert}}.
\end{eqnarray*}
Therefore, applying   \eqref{kp2}, the desired result is obtained.
\end{proof}

\renewcommand\bibname
  \end{document}